\documentclass[a4paper,11pt]{elsarticle}
\usepackage[utf8]{inputenc}

\usepackage{bm}
\usepackage{amsmath,amsthm,amssymb}
\usepackage{mathrsfs}
\usepackage{graphicx}
\usepackage{epsfig, setspace}
\usepackage{url}
\usepackage{graphicx, transparent, color}
\usepackage[ruled]{algorithm2e}
\usepackage{xcolor}
\usepackage{natbib}
\DeclareSymbolFont{rsfs}{U}{rsfs}{m}{n}
\DeclareSymbolFontAlphabet{\mathcal}{rsfs}
\makeatletter
\def\ps@pprintTitle{%
  \let\@oddhead\@empty
  \let\@evenhead\@empty
  \let\@oddfoot\@empty
  \let\@evenfoot\@oddfoot
}
\makeatother

\usepackage[margin=2.5cm]{geometry}
\usepackage{tikz}
\usetikzlibrary{shapes,calc}
\usepackage{verbatim}
\usepackage{listings}
\usepackage{textcomp}
\usepackage[normalem]{ulem}
\usepackage{todonotes}

\usepackage{amsfonts}
\usepackage{lmodern}
\usepackage{euscript}
\usepackage{oldgerm}

\numberwithin{equation}{section}

\newtheorem{theorem}{Theorem}[section]
\newtheorem{lemma}{Lemma}[section]
\newtheorem{proposition}[theorem]{Proposition}

\begin{document}
 \begin{frontmatter}
 \title{Existence and uniqueness of weak solutions to the Smoluchowski coagulation equation with source and sedimentation}
%\bibliographystyle{amsplain.bst}
%\shorttitle{DGFVEM Stokes Problems}
 \tnotetext[t1]{The authors are listed alphabetically}
\date{}
\author[1,2]{Prasanta Kumar Barik \corref{cor1}} \ead{barik@ugr.es,prasant.daonly01@gmail.com}
%\author[11]{}
\author[3]{Asha K. Dond}\ead{ashadond@iisertvm.ac.in}
\author[3]{Rakesh Kumar}\ead{rakeshiitb21@gmail.com, rakeshmath21@iisertvm.ac.in}
\cortext[cor1]{Corresponding author Tel.: +91 80 6695 3731}
\address[1]{Instituto de Matemáticas,\\
Universidad de Granada,
 Rector López Argüet, 
S/N, 18001 Granada,
Spain}
\address[2]{Departamento de Matemática Aplicada,\\
Universidad de Granada,
Avenida de Fuentenueva S/N, 18071 Granada, Spain}

\address[3]{School of Mathematics,\\
Indian Institute of Science Education and Research
Thiruvananthapuram, \\
Thiruvananthapuram-695551, India.}
% \maketitle

\begin{abstract}
%\cite{zhu-etal_09}.
This article is devoted to a generalized version of Smoluchowski's coagulation equation. This model describes the time evolution of a system of aggregating particles under the effect of external input and output particles. We show that for a large class of coagulation kernels, output rates, and exponentially decaying input rates, there is a weak solution. Moreover, the solution satisfies the mass-conservation property for linear coagulation rate and an additional condition on input and output rates. The uniqueness of weak solutions  is also established by applying additional restrictions on the rates. 
\end{abstract}
\end{frontmatter}

\section{Introduction}
\noindent Let us first recall that the Smoluchowski coagulation equation is a mean-field model that describes how small particles merge together to form bigger ones. These types of models are used in different fields of science. For example,  the formation of stars and planets in astrophysics, polymerization in chemistry, animal grouping in biology, etc. The Smoluchowski coagulation model was first introduced by Marino Smoluchowski \cite{Smoluchowski:1917} in his seminal work  by considering each particle size as a positive integer to describe the coagulation of colloids moving according to a Brownian motion. By treating each particle size as a positive real number, M\"{u}ller \cite{Muller:1928} proposed the continuous version of the Smoluchowski coagulation equation (SCE) in the form of a partial integro-differential equation and which is also known as the \textit{Smoluchowski coagulation equation}. Let $\zeta(t, p) \ge 0$ be the concentration of particles of mass $p \in \mathbb{R}_{>0}:= (0, \infty)$ at time $t \ge 0$, and then the dynamics of $\zeta $ is given by \cite{Barik:2018, Laurencot:2015} and which reads   
\begin{align}\label{SCE}
\frac{\partial \zeta(t, p) }{\partial t}  =  \mathcal{C}_a ( \zeta(t, p) ) - \mathcal{C}_d ( \zeta(t, p) ),
\end{align}
where
\begin{align*}
 \mathcal{C}_a ( \zeta(t, p) ) := \frac{1}{2} \int_{0}^{p} \mathcal{A}(p-q, q) \zeta(t, p-q) \zeta(t, q) dq,    
\end{align*}
and
\begin{align*}
 \mathcal{C}_d ( \zeta(t, p) ) :=  \int_{0}^{\infty} \mathcal{A}(p, q) \zeta(t, p) \zeta(t, q) dq.    
\end{align*}
The familiar terms $ \mathcal{C}_a ( \zeta(t, p) )$ and $ \mathcal{C}_d ( \zeta(t, p) )$  in \eqref{SCE} describe the appearance of particles of size $p$ after coalescing particles of size $q$ and $p-q$ and the disappearance of particles of size $p$ after combining with particles of size $q$ due to the coagulation process, respectively. In addition, the non-negative and symmetric function $\mathcal{A}(p, q)$ is known as the \textit{coagulation rate} or kernel and describes the rate at which particles of size $p$ unite with a particle of size $q$ to create a larger one of size $p+q$.

In this article, we study another mathematical model which is an extension model of the Smoluchowski coagulation equation \eqref{SCE}, see \cite{Ananda:2012, Foret:2012, Makoveeva:2022}. This type of model has many applications in science such as chemistry, cloud physics, and oceanography. If $\zeta(t, p) $ be the concentration of particle of size $p$ at time $t$, then this model can be expressed as: 
\begin{align}\label{SCEs}
\frac{\partial \zeta(t, p)}{\partial t}  =  \mathcal{C}_a ( \zeta(t, p) ) - \mathcal{C}_d ( \zeta(t, p) )  + \mathcal{S}(p) -  \mathcal{R}(p) \zeta(t, p),
\end{align}
with the initial value
\begin{align}\label{Initialdata}
\zeta(0, p) = \zeta^{in}(p)\ge 0~ \mbox{a.e.}
\end{align}
We can observe that the first two terms on the right-hand side of \eqref{SCEs} are the same as those in \eqref{SCE}. The term $\mathcal{S}(p)$ describes the rate at which the particles of size $p$ are injected into the system. We assume that the injection or source term decays faster for large-size particles. On the right-hand side to \eqref{SCEs}, the last term is known as the removal term, which removes the particle of size $p$ from the system at the rate $\mathcal{R}(p) \zeta(p)$. In addition,  $\mathcal{R}(p)$ is known as the removal co-efficient. From the application point of view, the term $\mathcal{R}(p) \zeta(p)$ is also known as the sedimentation of particles due to gravity and in chemistry, this term also describes the removal of product crystals when considering bulk crystal growth in supercooled melts and supersaturated solutions \cite{Alexandrov:2022, Buyevich:1994}.
Next, the total mass of particles for the generalized coagulation equation \eqref{SCEs} can be defined as:
 \begin{align}\label{Totalmass}
\mathcal{M}_1(t)=\mathcal{M}_1(\zeta(t)):=\int_0^{\infty} p \zeta(t, p)\ dp,  \ \ t \ge 0.
\end{align}
In general, for the linear coagulation rate, the coagulation equations conserved the mass. However, due to the source  and removal terms, we expect the total mass in the system does not conserve. The failures of mass conservation property in physics literature is known as the \textit{gelation process} and the finite time at which this process occurs is known as the \emph{gelation time} \cite{Escobedo:2003, Leyvraz:1981}. The solution $\zeta$  may hold the mass-conservation property for some special cases.

 Prior to entering into the details of this work, let us first look at the available literature related to the coagulation-fragmentation model (CF) and the generalized coagulation model. The existence, uniqueness, and asymptotic behavior of solutions to CF equations have been discussed in many articles by considering different growth conditions on the kernels, see \cite{Stewart:1989, Stewart:1990, Barik:2018, Barik:2021, Giri:2012, Laurencot:2015}. However, the well-posedness of weak solutions to \eqref{SCEs}--\eqref{Initialdata} have not been discussed yet. Even though there are a few articles in which different types of solutions for the discrete version of the model \eqref{SCEs} have been addressed, see \cite{Kuehn:2019, White:1982, Crump:1982, Hayakawa:1987, Hendriks:1984, Hendriks:1985, Lushnikov:1976}. More precisely, the existence of steady-state solutions has been discussed in \cite{White:1982, Crump:1982, Hayakawa:1987, Hendriks:1984, Hendriks:1985}. In \cite{Kuehn:2019}, the authors recently discussed the existence and uniqueness of solutions to \eqref{SCEs}--\eqref{Initialdata} for a large class of coagulation rates whereas the removal coefficients grow sufficiently fast for large particle sizes. In addition, they have also addressed the equilibrium solutions. However, the corresponding continuous models have not been studied up to that level in the mathematical community. Even though there are a few articles that are based on the Smoluchowski coagulation equation with source, see \cite{Laurencot:2020, Ferreira:2022, Ferreira:2023}. In \cite{Laurencot:2020}, the existence of a stationary solution is discussed. In \cite{Cristian:2022}, the non-existence of stationary solutions has been discussed for some special cases. The existence of self-similar solutions has been discussed by the authors in \cite{Ferreira:2022}. Recently, non-equilibrium solutions is discussed by authors for the multi-component coagulation equations with source term, see \cite{Ferreira:2023}.   To the best of our knowledge, this is the first attempt to discuss rigorously the well-posedness of weak solutions to the Smoluchowski coagulation with source and sedimentation for both linear and product-type coagulation kernels.

Let us end this section by describing the content of the paper. We present some preliminary results, assumptions, and main results in Section 2. In Section 3, the existence and uniqueness of approximated solutions to \eqref{SCEs}--\eqref{Initialdata} is shown by using the Picard-Lindel\"{o}f theorem. Also, we prove the existence of weak solutions by using a weak $L^1$ compactness technique in this section. In addition, we also show that the solution satisfies the mass conservation property in this section. An analysis of the uniqueness of the solution is demonstrated in the final section.

%%%%%%%%%%%%%%%%%%%%%%%%%%%%%%%%%%%%%%%%
%%%%%%%%%%%%%%%%%%%%%%%%%%%%%%%%%%%%%%%%
\section{Assumptions, Preliminaries results and the Statement of the Main Results}
%%%%%%%%%%%%%%%%%%%%%%%%%%%%%%%%%%%%%%%%
%%%%%%%%%%%%%%%%%%%%%%%%%%%%%%%%%%%%%%%%
 Let us begin this section by describing the assumptions on the coagulation rate $\mathcal{A}$, input term $\mathcal{S}$, reaction rate $\mathcal{R}$, and the initial data $\zeta^{in}$. 
We assume that the coagulation kernel $\mathcal{A}$ is a non-negative and symmetric measurable function that satisfies
 $\mathcal{A}(p, q) = \mathcal{A}(q, p)$, for all $(p, q) \in \mathbb{R}_{>0} \times \mathbb{R}_{>0}$.
 In addition, it also satisfies the following class of growth conditions for some non-negative constant $A$:\\
Class-I: linear kernel:
 \begin{equation}\label{SumK}
\mathcal{A}(p, q):=\begin{cases}
A,\       & \text{if}\ p+q < 1,\ \\
A (p+q),\ &  \text{if}\ p+q > 1,
\end{cases}
\end{equation}
or\\
Class-II: Product kernel: for some function $\eta: (0, \infty) \rightarrow [1, \infty)$
 \begin{equation}\label{ProductK}
\mathcal{A}(p, q):=\begin{cases}
A,\ & \text{if}\ (p, q) \in (0,1)^2,\ \\
A \eta(q), \ &  \text{if}\ (p, q) \in (0,1) \times (1, \infty), \  \\
A \eta(p), \ &  \text{if}\ (p, q) \in  (1, \infty) \times (0, 1),\ \\
A \eta(p) \eta(q), \  &  \text{if}\ (p, q) \in (1, \infty)^2.
\end{cases}
\end{equation}
where $\eta$ satisfies the following conditions:
\begin{align}\label{RateCoag}
\eta^{\ast} :=\sup_{p \ge 1} \frac{\eta(p)}{(1+p)} < \infty\ \ \text{and} \ \lim_{p \to \infty} \frac{\eta(p)}{p} =0.  
\end{align}
From \eqref{SumK}, \eqref{ProductK} and \eqref{RateCoag}, we infer that
\begin{align}\label{EstimateCoagRate}
\mathcal{A}(p, q) \le A (1+p)(1+q).    
\end{align}
Let us assume that the injection term $\mathcal{S}$ is a non-negative and exponentially decreasing function that satisfies:
\begin{equation}\label{RateInjection}
\mathcal{S} \in L^1_{0, 1}(\mathbb{R}_{>0}).
\end{equation} 
Furthermore, in the case of the linear coagulation rate, we assume that the removal coefficient $\mathcal{R}(p)$ is any measurable function, while for the product kernel case, the removal coefficient $\mathcal{R}(p)$ is an increasing function and which satisfies
\begin{equation}\label{RateRemovable}
\mathcal{R}(p) \leq k (1+p)^{\alpha } ,\ \forall p \in \mathbb{R}_{>0},
\end{equation}
where $ k \ge 0$ and $\alpha \in [0, 1)$.\\
Finally, we assume that the initial data $\zeta^{in}$ satisfies the following condition:
\begin{equation}\label{ConditionIntial}
\zeta^{in} \in L^1_{0, 1}(\mathbb{R}_{>0}).
\end{equation}

Now, we are in a position to state the main theorems of this paper.
\begin{theorem}\label{TheoremSCEs}
 Consider a function $\zeta^{in}$ satisfying \eqref{ConditionIntial}. In addition, $\mathcal{S}$ and $\mathcal{R}$ enjoy assumptions  \eqref{RateInjection} and \eqref{RateRemovable}, respectively.\\
 $(i)$ \ \ If the function $\mathcal{A}$ satisfies \eqref{SumK}. Then there exists at least one weak solution $\zeta$  to \eqref{SCEs}--\eqref{Initialdata} such that
 \begin{align}\label{TheoremEquation1}
  \zeta \in \mathcal{C}([0, \infty)_w; L^1_{0, 1}(\mathbb{R}_{>0} ) ) \cap L^{\infty}( (0, T);   L_{0, 1}^1(\mathbb{R}_{>0} ) )  \ \text{for each}\ T>0.
 \end{align}
 where $\mathcal{C}([0, \infty)_w; L^1_{0, 1}(\mathbb{R}_{>0} ) )$ denotes the space of all weakly continuous functions from $[0, \infty)$ to $L^1_{0, 1}(\mathbb{R}_{>0} )$. More precisely, a function $g \in \mathcal{C}([0, \infty)_w; L^1_{0, 1}(\mathbb{R}_{>0} ) )$ if  
 \begin{align*}
     t \mapsto \int_0^{\infty } \omega(p) (1+p) g (t, p) dp 
 \end{align*}
 is continuous on $[0, \infty)$ for all $\omega \in L^{\infty}(\mathbb{R}_{>0} ) $.\\
 Furthermore, if 
 \begin{align}\label{Cond R S}
     \int_0^{\infty} p \mathcal{S}(p) dp =   \int_0^{\infty} p \mathcal{R}(p) \zeta(s, p)  dp,
 \end{align}
 then the solution $\zeta$ satisfies the mass conservation property, i.e.,
  \begin{align}\label{MCPS}
   \mathcal{M}_1(t) =   \int_0^{\infty} p \zeta^{in}(p)  dp.
 \end{align}
 $(ii)$ \ \ If the function $\mathcal{A}$ satisfies \eqref{ProductK}. Then there exists at least one weak solution $\zeta$  to \eqref{SCEs}--\eqref{Initialdata} such that
 \begin{align}\label{TheoremEquation2}
  \zeta \in \mathcal{C}([0, \infty)_w; L^1(\mathbb{R}_{>0} ) )  \cap L^{\infty}( (0, T);   L_{0, 1}^1(\mathbb{R}_{>0} ) ) \ \text{for each}\ T>0.
 \end{align}
 
 Furthermore, the solution $\zeta$ in both cases satisfies the following weak formulation
 \begin{align}\label{definition}
\int_0^{\infty} [ \zeta(t, p) - \zeta^{in}(p)] \omega(p) dp=&\frac{1}{2}\int_0^t \int_0^{\infty} \int_{0}^{\infty}\tilde{\omega}(p, q) \mathcal{A}(p, q) \zeta(s, p) \zeta(s, q) dq dp ds\nonumber\\
&+\int_0^t \int_0^{\infty} {\omega}(p)[ \mathcal{S}(p) - \mathcal{R}(p) \zeta(s, p)] dp ds,
\end{align}
where
\begin{align}\label{Identity}
\tilde{\omega} (p, q):=\omega(p+q)-\omega (p)-\omega(q)
\end{align}
for every $t \in [0, T]$ and $\omega \in L^{\infty}(\mathbb{R}_{>0})$.
 \end{theorem}

\begin{theorem}\label{UniquenessTheoremCRBK}
Let $\zeta$ be a weak solution to \eqref{SCEs}--\eqref{Initialdata}. If the coagulation equation satisfies 
\begin{align}\label{UniqueThm}
\mathcal{A}(p, q) \le A (1+p)^{1/2}(1+q)^{1/2},\ \   \text{for all} \ \  (p, q) \in \mathbb{R}_{>0} \times \mathbb{R}_{>0},\ \text{for some} \ A \ge 0,
\end{align}
with the same source and removal coefficient as defined in Theorem \ref{TheoremSCEs}.  Then \eqref{SCEs}--\eqref{Initialdata} admits a unique weak solution.
 \end{theorem}

The strategy of the proof of the Theorem \ref{TheoremSCEs} is based on the classical $L^1$ weak compactness technique by applying the  Dunford-Pettis theorem and a variant of the Arzel\`{a}-Ascoli theorem, see \cite{Vrabie:1995}. For this, we assume additional integrability conditions on $\zeta^{in}$ and $\mathcal{S}$ by applying a refined version of de la Vall\'{e}e-Poussin theorem see  \cite[Theorem~2.8]{Laurencot:2015} which are described below.  More precisely, let us first construct two non-negative, convex, and non-decreasing functions $\sigma_1, \sigma_2  \in \mathcal{C}^{1}([0, \infty))$ such that
\begin{description}
  \item[(i)] $\sigma_j(0)=\sigma_j'(0)=0$ and $\sigma_j'$ is concave;
  \item[(ii)] $\lim_{p \to \infty} \sigma_j'(p) =\lim_{p \to \infty} \frac{ \sigma_j(p)}{p}=\infty$
  for $j=1, 2$.
\end{description}
Let us now recall a refined version of De la Vall\'{e}e-Poussin theorem for $p=1$ see  \cite[Theorem~2.29]{Fonseca:2007}.
\begin{theorem}\label{TheoremDelaVallee}
Let $(X, \mu )$ be a measure space, and let $\mathcal{D} \subset L^1(X)$ be a bounded set. Then the following statements are equivalent:\\
$(a.)$ $\mathcal{D} $ is equi-integrable;\\
$(b.)$ there exists an increasing and convex function $\sigma : (0, \infty ) \mapsto [0, \infty ]$, with 
\begin{align*}
    \lim_{p \to \infty } \frac{\sigma(p)}{p} = \infty \ \ \text{such that}
\end{align*}
\begin{align*}
    \sup_{\zeta \in  \mathcal{D}  }  \int_X \sigma( |  \zeta |  ) d\mu < \infty.
\end{align*}
\end{theorem}

Since  $\mathcal{S} \in L_{0, 1}^1(\mathbb{R}_{>0})$, then applying the above-refined version of de la Vall\'{e}e-Poussin theorem, we ensure that there exist two non-negative convex functions $\sigma_1$ and $\sigma_2$ such that 
\begin{align}\label{Convexp1}
\sigma_i(0)=0,~~~\lim_{p \to {\infty}}\frac{\sigma_i(p)}{p}=\infty,~~~~i=1,2,
\end{align}
\begin{align}\label{Convexp3}
\Gamma_3 := \int_0^{\infty}\sigma_1(p) \mathcal{S}(p) dp <\infty,~~\text{and}~~\Gamma_4 :=\int_0^{\infty}{\sigma_2( \mathcal{S} (p))} dp < \infty.
\end{align}
Using the above-refined version of de la Vall\'{e}e-Poussin theorem  once more for $\zeta^{in} \in L_{0, 1}^1(\mathbb{R}_{>0})$, then it can easily be seen that there also exist similar non-negative convex functions $\sigma_1$ and $\sigma_2$ such that both  $\sigma_1$ and $\sigma_2$ satisfy \eqref{Convexp1}, and
\begin{align}\label{Convexp2}
\Gamma_1 := \int_0^{\infty}\sigma_1(p) \zeta^{in}(p) dp< \infty,~~\text{and}~~\Gamma_2 :=\int_0^{\infty}{\sigma_2(\zeta^{in}(p))} dp<\infty.
\end{align}
Let us recall some additional properties of the above convex functions $\sigma_1$ and $\sigma_2$ from \cite{Laurencot:2015}. 
\begin{lemma} Consider $\sigma_1$, $\sigma_2$ in $\mathcal{C}^1{ ( [0, \infty ))  }$ such that $\sigma'_1$ and $\sigma'_2$ are concave. Then, we have
\begin{equation}\label{Convexp4}
\hspace{-5cm} \sigma_2(p_1)\le p_1\sigma'_2(p_1)\le 2\sigma_2(p_1),
\end{equation}
\begin{equation}\label{Convexp5}
\hspace{-5.5cm} p_1\sigma_2'(p_2)\le \sigma_2(p_1)+\sigma_2(p_2),
\end{equation}
and
\begin{equation}\label{Convexp6}
0 \le \sigma_1(p_1+p_2)-\sigma_1(p_1)-\sigma_1(p_2)\le  2\frac{p_1\sigma_1(p_2)+p_2\sigma_1(p_1)}{p_1+p_2},
\end{equation}
 for all $p_1, p_2 \in \mathbb{R}_{>0}$.
\end{lemma}

%%%%%%%%%%%%%%%%%%%%%%%%%%%%%%%%%%%%%%
%%%%%%%%%%%%%%%%%%%%%%%%%%%%%%%%%%%%%%
\section{Existence of weak solutions}
%%%%%%%%%%%%%%%%%%%%%%%%%%%%%%%%%%%%%%
%%%%%%%%%%%%%%%%%%%%%%%%%%%%%%%%%%%%%%
In this section, our aim is to prove the Theorem \ref{TheoremSCEs}. More precisely, the strategy to prove the Theorem \ref{TheoremSCEs} is achieved in two steps. Step-I: For each $n \in \mathbb{N}$, find a unique classical solution to the following truncated equations:  
\begin{align}\label{TSCEs}
\frac{\partial \zeta_n(t, p)}{\partial t}  = &\mathcal{C}^n_a(\zeta_n (t, p) )  - \mathcal{C}^n_d(\zeta_n (t, p) )  + \mathcal{S}_n(p) - \mathcal{R}_n(p) \zeta_n(t, p),
\end{align}
where
\begin{align*}
\mathcal{C}_a^n(\zeta_n (t, p) ) :=  \frac{1}{2} \int_{0}^{p} \mathcal{A}_n(p-q, q) {\zeta_n}(t, p-q) {\zeta_n}(t, q) dq,
\end{align*}
\begin{align*}
\mathcal{C}_d^n(\zeta_n (t, p) ) := \int_{0}^{n } \mathcal{A}_n(p, q) \zeta_n(t, p) \zeta_n(t, q) dq,
\end{align*}
\begin{equation}\label{Ckterm}
\mathcal{A}_n(p, q) := \mathcal{A}(p, q) \chi_{(0, n)}(p)  \chi_{(0, n)}(q),
\end{equation}
\begin{align}\label{Sterm}
\mathcal{S}_n(p) := \mathcal{S}(p) \chi_{(0, n )}(p),
\end{align}
and
\begin{align}\label{Rterm}
\mathcal{R}_n(p) := \mathcal{R}(p) \chi_{(0, n )}(p),
\end{align}
with the truncated initial condition
\begin{align}\label{TInitial}
\zeta_n(0, p)= \zeta_n^{in}(p),  ~~ \text{for}~~ p\in (0, n).
\end{align}
Here $\chi_{B}$ denotes the characteristic function which is defined as:
\begin{equation*}
\chi_{B} (p):=\begin{cases}
1,\       &  \text{if}\ p \in  B,\ \\
0,\       &  \text{if}\ p \notin B.
\end{cases}   
\end{equation*}

Step-II: A weak $L^1$ compactness technique is applied to the family of solutions $\{ \zeta_n \}_{n >1}$ to obtain the weak solution to \eqref{SCEs}--\eqref{Initialdata}.
\vspace{.5cm}

%%%%%%%%%%%%%%%%%%%%%%%%%%%%%%%%%%%%%%
%%%%%%%%%%%%%%%%%%%%%%%%%%%%%%%%%%%%%%
\textbf{Step-I}: The following proposition solves Step-I.
\begin{proposition}\label{Prop1}
Let $n > 1$. Then, \eqref{TSCEs}, \eqref{TInitial} has a unique non-negative solution $\zeta_n\in \mathcal{C}^1([0,\infty);L^1(0,n))$. In addition, it satisfies
\begin{align}\label{PropMassbound}
 \frac{ d   }{ dt} \int_0^n p \zeta_n(t, p) dp = & \int_0^n \int_{n-p}^n  p\mathcal{A}_n(p, q) {\zeta_n}(t, p) {\zeta_n}(t, q) dq dp \nonumber\\
 & +\int_0^n p \mathcal{S}_n(p) dp -  \int_0^n p\mathcal{R}_n(p) \zeta_n(t, p) dp, 
\end{align}
for $t\ge 0$.
\end{proposition}

\begin{proof}
We wish to prove this proposition by applying the Picard-Lindel\"{o}f theorem or Cauchy-Lipschitz theorem \cite[Theorem 7.3]{Brezis:2011} in the Banach space $L^1(0, n) $. On the one hand, from \eqref{Ckterm} and \eqref{EstimateCoagRate}, we estimate
 \begin{align}\label{bound for kernel}
\mathcal{A}_n(p, q) \le 4 A\ n^2,\ \ \text{for }\ \ n > 1.
 \end{align}
On the other hand, we show that each term in the right-hand side of \eqref{Ckterm} is locally Lipschitz continuous in the space $L^1(0, n)$. \\
For this purpose, let $\zeta^1 \text{and}\ \zeta^2 \in L^1(0, n) $. Then, from \eqref{Ckterm}, Fubini's theorem and the symmetry of $ \mathcal{A}_n$, we evaluate
\begin{align}\label{Lipschitz1}
  \| \mathcal{C}_a^n(\zeta^1) - &  \mathcal{C}_a^n(\zeta^2)  \|_{L^1(0, n)}  \nonumber\\
& \le 2 A n^2 (\| \zeta^1\|_{L^1(0, n)}+ \|\zeta^2\|_{L^1(0, n)}) \|\zeta^1 - \zeta^2\|_{L^1(0, n)},
\end{align}
\begin{align}\label{Lipschitz2}
\|  \mathcal{C}_{d}^n(\zeta^1)- &  \mathcal{C}_{d}^n(\zeta^2) \|_{L^1(0, n) } \nonumber\\
&\le 4 A n^2 (\|\zeta^1\|_{L^1(0, n)}+ \|\zeta^2\|_{L^1(0, n) }) \|\zeta^1 - \zeta^2\|_{L^1(0, n)},
\end{align}
\begin{align}\label{Lipschitz3}
\|  \mathcal{R}_n(p) ( \zeta^1- \zeta^2 )  \|_{L^1(0, n) }  \le  2k n  \| \zeta^1 - \zeta^2\|_{L^1(0, n)}.
\end{align}
Hence, $ \mathcal{C}_{a}^n(\zeta_n)$, $ \mathcal{C}_{d}^n(\zeta_n)$ and $\mathcal{R}_n(p) \zeta_n(p)$ are locally Lipschitz continuous functions in $L^1(0, n)$. Thus, \eqref{TSCEs}, \eqref{TInitial} has a unique solution $\zeta_n \in \mathcal{C}^1([0, \mathcal{T} ) ; L^1(0, n) )$  according to the Picard-Lindel\"{o}f theorem or Cauchy-Lipschitz theorem, see \cite[Theorem 7.3]{Brezis:2011}. Moreover, this solution is defined on a maximal interval $t \in [0, \mathcal{T} )$, $\mathcal{T} \in (0, \infty]$, and either $\mathcal{T} =\infty$ or
\begin{align}\label{norminfinity}
\mathcal{T} < \infty\ \ \text{and} \ \ \lim_{t \to \mathcal{T}  } \| \zeta_n(t)\|_{L^1(0, n) } =\infty.
\end{align}
\textbf{Positivity of solutions:}\\
Since $\mathcal{C}_{a}^n(\zeta_n)$ is a locally Lipschitz continuous functions in $L^1(0, n)$. Thus, the positive part of $[ \mathcal{C}_{a}^n (\zeta_n) ]_{+}$ is also a locally Lipschitz continuous. Similar to previous arguments, we can show that the following initial value problem has also a unique solution
\begin{align}\label{TCSTpositive}
\frac{\partial \zeta_n}{\partial t}  = [ \mathcal{C}_a^n( \zeta_n)]_{+} -  \mathcal{C}_d^n(\zeta_n)  + \mathcal{S}_n - \mathcal{R}_n \zeta_n,
\end{align}
with the same initial data given in \eqref{TInitial}.\\ Let $sign_+(r)=1$, for $r \ge 0$ and $sign_+(r)=0$, for $r < 0$. Then, from \eqref{bound for kernel}, \eqref{TCSTpositive} and $ (-\zeta)_{+} =- sign_{+} (-\zeta)  ( \zeta ) $, we infer that
\begin{align}\label{Truncated1}
& \frac{d}{dt} \| (-\zeta_n)_{+} (t)\|_{L^1(0, n) } \nonumber\\
 = & -\int_0^n sign_{+}(-\zeta_n)(t, p) ([ \mathcal{C}_a^n(\zeta_n (t, p) )]_{+} -  \mathcal{C}_d^n(\zeta_n(t, p)) \nonumber\\
 & +  \mathcal{S}_n(p) - \mathcal{R}_n(p) \zeta_n(t, p)) dp \nonumber\\
\le & \int_0^n \int_0^n sign_{+}(-\zeta_n)(t, p) \mathcal{A}_n(p, q) \zeta_n(t, p) \zeta_n(t, q) dq   dp \nonumber\\
& + \int_0^n sign_{+}(-\zeta_n)(t, p) \mathcal{R}_n(p) \zeta_n(t, p)    dp \nonumber\\
\le & -\int_0^n \int_0^n [- \zeta_n(t, p)]_{+} \mathcal{A}_n(p, q) \zeta_n(t, q) dq  dp -\int_0^n  [- \zeta_n(t, p)]_{+} \mathcal{R}_n(p)  dp \nonumber\\
\le & 4 A n^2 \| \zeta_n(t)\|_{L^1(0, n) } \| (-\zeta_n)_{+}(t) \|_{L^1(0, n)} + 2 k n  \| (-\zeta_n)_{+}(t) \|_{L^1(0, n)}.
\end{align}
After solving the above differential inequality, we obtain
\begin{align}\label{Truncated2}
 \| (-\zeta_n)_{+} & (t)\|_{L^1(0, n )} \nonumber\\
 & \le  \bigg(  \| (-\zeta_n^{in} )_{+} \|_{L^1(0, n)}      \bigg) \exp \bigg( 2 k n t + 4 A n^2 \int_0^t  \| \zeta_n(s)\|_{L^1(0, n)} ds \bigg).  
\end{align}
Using the non-negativity of \eqref{TInitial} and the positive part of $(-\zeta_n^{in})_{+}$ into \eqref{Truncated2}, we get
\begin{align}\label{Truncated3}
 \| (-\zeta_n)_{+} (t)\|_{L^1(0, n) } \le 0.
\end{align}

Finally, we can infer from \eqref{Truncated3} that $ \zeta_n(\cdot, t) \ge 0$ $\forall t \in  [0, \mathcal{T})$. Therefore, both  \eqref{TCSTpositive} and \eqref{TSCEs} are the same equations. Furthermore, $\zeta_n$ satisfies \eqref{PropMassbound}  which can easily be shown by multiplying $p$ into \eqref{TSCEs} and then using the Fubini theorem and the symmetry of $\mathcal{A}_n$ for $t \in  [0, \mathcal{T} )$. Now, our claim is $\mathcal{T}= \infty $. For this, we evaluate the following integral by using 
 \eqref{TSCEs}, the repeated applications of Fubini's theorem, the transformation $p-q=p'$ \& $q=q'$, the symmetry of $\mathcal{A}_n$, as 
 \begin{align}\label{Truncated4}
\frac{d}{dt} \int_0^n  \zeta_n(t, p) dp
= & - \frac{1}{2}\int_0^n \int_{0}^{n-p}   \mathcal{A}_n(p, q) \zeta_n(t, p) \zeta_n(t, q) dq dp\nonumber\\
& - \int_0^n \int_{n-p}^n   \mathcal{A}_n(p, q) \zeta_n(t, p) \zeta_n(t, q) dq dp\nonumber\\
& + \int_0^n    \mathcal{S}_n(p) dp  -  \int_0^n   \mathcal{R}_n(p) \zeta_n(t, p) dp \nonumber\\
\le  & \int_0^n    \mathcal{S}_n(p) dp  \le \|\mathcal{S}\|_{L^1_{0, 1} (\mathbb{R}_{>0} ) }.
\end{align}
 Using the non-negativity of $\zeta_n(\cdot, t)$ $\forall t \in  [0, \mathcal{T})$ 
and the above solving differential inequality to \eqref{Truncated4}, we obtain
  \begin{align}\label{Truncated6}
 \| \zeta_n( t) \|_{L^1(0, n)}
\le & \bigg( \| \zeta_n^{in} \|_{L^1(0, n)} + \|   \mathcal{S}  \|_{L^1_{0, 1} (\mathbb{R}_{>0} )  } t \bigg), \ \ \forall \ t\in [0, \mathcal{T}).
\end{align}
From \eqref{Truncated6}, one can see that if $ \mathcal{T} < \infty$, then
\begin{align*}
 \lim_{t \to \mathcal{T} } \| \zeta_n(t)\|_{L^1(0, n)} \le  \| \zeta_n^{in} \|_{L^1(0, n)}  + \|     \mathcal{S}     \|_{L^1_{0, 1}  (\mathbb{R}_{>0} ) }  \mathcal{T} < \infty.
\end{align*}
Thus, it does not satisfy \eqref{norminfinity}. Consequently, $\mathcal{T}= \infty$.
 This completes the proof of Proposition \ref{Prop1}.
\end{proof}
%%%%%%%%%%%%%%%%%%%%%%%%%%%%%%%%%%%%%
%%%%%%%%%%%%%%%%%%%%%%%%%%%%%%%%%%%%%

%%%%%%%%%%%%%%%%%%%%%%%%%%%%%%%%%%%%%
%%%%%%%%%%%%%%%%%%%%%%%%%%%%%%%%%%%%%
Let us now state a crucial identity that may help to prove some important results in our analysis.
\begin{lemma}\label{LemmaIdentity}
For each $n > 1$, let $\zeta_n$ be a solution to \eqref{TSCEs}, \eqref{TInitial}. Then for $\omega \in L^{\infty}(\mathbb{R}_{>0})$, $\zeta_n$ satisfies the following identity
\begin{align}\label{TVWSolution}
&\int_0^n [ \zeta_n(p, t)- \zeta_n^{in}(p) ] \omega(p) \ dp\nonumber\\
=& \frac{1}{2} \int_0^t \int_0^n \int_0^{n} H_{\omega, n}(p, q)  \mathcal{A}_n(p, q) \zeta_n(s, p) \zeta_n(s, q) d q  d p ds\nonumber\\
& +\int_0^t\int_0^n \omega(p) \mathcal{S}_n(p) d p ds - \int_0^t\int_0^n \omega(p) \mathcal{R}_n(p) \zeta_n(s, p)  dp ds,
\end{align}
where
\begin{align}\label{HOmega}
 H_{\omega, n}(p, q)= \omega (p+q) \chi_{(0, n)}(p+q)- [\omega(p) + \omega(q)].
 \end{align}
\end{lemma}

\textbf{Step-II:}
In this step, our aim is to show that the family of solutions $\{ \zeta_n\}_{n > 1}$ is relatively compact in $\mathcal{C}([0, T]_w; L_{0, 1}^1(\mathbb{R}_{>0} ) )$ for linear coagulation kernel and $\{ \zeta_n\}_{n > 1}$ is relatively compact in $\mathcal{C}([0, T]_w; L^1(\mathbb{R}_{>0} ) )$ for product-type of kernels. For that purpose, we implement a weak $L^1$ compactness technique which is used in the pioneering work of Stewart \cite{Stewart:1989}. In the next lemma, we show the family of solutions $\{ \zeta_n\}_{n > 1}$ is uniformly bounded in $L_{0, 1}^1(\mathbb{R}_{>0})$.

\subsection{Uniform Bound}
\begin{lemma}\label{Uboundlemma}
 Let $T>0$.  Consider $\zeta^{in}$ satisfies \eqref{ConditionIntial} and assume that the coagulation rate $\mathcal{A}$ satisfies  \eqref{EstimateCoagRate}. Suppose $\mathcal{S}$ and $\mathcal{R}$ enjoy assumptions  \eqref{RateInjection} and \eqref{RateRemovable}, respectively. Then, we have
\begin{align*}
\int_0^{n} (1+ p)\ \zeta_n(t, p) dp \le \Lambda(T)\ \ \text{for all}\ t \in [0, T],
\end{align*}
where $\Lambda(T)$ is a positive constant depending on $T$.

\begin{align*}
(ii)\ \int_0^t \int_0^{n} \mathcal{R}_n(p)  \zeta_n(\tau, p) dp d\tau \le \|\zeta^{in} \|_{L^1_{0, 1}(\mathbb{R}_{>0})  } + \|\mathcal{S}\|_{L^1_{0, 1}(\mathbb{R}_{>0})} t, \ \text{for all}\ t \in [0, T].
\end{align*}

\end{lemma}
\begin{proof}
Consider $\omega \equiv 1$. From \eqref{HOmega}, we have $H_{\omega, n}  \le -1 $. Then, it can be inferred from \eqref{LemmaIdentity} and \eqref{TVWSolution} that
\begin{align}\label{bound1}
 \int_0^n   \{ \zeta_n(t, p)- \zeta_n^{in}(p) \} \ dp \le &  - \frac{1}{2} \int_0^t \int_0^n \int_0^n  \mathcal{A}_n(p, q) \zeta_n(s, p) \zeta_n(s, q) dq dp ds\nonumber\\
&    + \int_0^t\int_0^n \bigg( \mathcal{S}_n(p)  -  \mathcal{R}_n(p) \zeta_n(s, p) \bigg)  dp ds, 
\end{align}
On the one hand, using the non-negativity of $\zeta_n$, $\mathcal{A}_n$ and $\mathcal{R}_n$ and on the other hand, applying the integrability condition \eqref{RateInjection} into \eqref{bound1}, we estimate
\begin{align*}
 \int_0^n  \zeta_n(t, p)dp 
 \le & \int_0^n  \zeta_n^{in}(p)dp + \int_0^t \int_0^n \mathcal{S}_n(p) dp ds \\
\le & \|\zeta^{in} \|_{L^1_{0, 1}(\mathbb{R}_{>0})  } + \|\mathcal{S}\|_{L^1_{0, 1}(\mathbb{R}_{>0})} t \le \Lambda^{\ast}(T),
\end{align*}
where $\Lambda^{\ast}(T):= \|\zeta^{in} \|_{L^1_{0, 1}(\mathbb{R}_{>0})  }+ \| \mathcal{S} \|_{L^1_{0, 1}(\mathbb{R}_{>0})} T$. This proves the Lemma \ref{Uboundlemma} $(ii)$.
Finally, using \eqref{PropMassbound}, we evaluate
\begin{align*}
\int_0^n (1+p) \ \zeta_n(t, p)\ dp \le \Lambda(T),
\end{align*}
where $\Lambda(T):= \Lambda^{\ast}(T) + \int_0^{\infty} p \zeta^{in}(p) dp + \| \mathcal{S} \|_{L^1_{0, 1}(\mathbb{R}_{>0} ) }$, for each $n>1$.
This proves the Lemma \ref{Uboundlemma} $(i)$.
\end{proof}

We next discuss the behavior of truncated solution $\zeta_n$ for large size particle $p$ when the coagulation rate satisfies \eqref{SumK}.
%%%%%%%%%%%%%%%%%%%%%%%%%%%%%%%%%%%%%%%%%
%%%%%%%%%%%%%%%%%%%%%%%%%%%%%%%%%%%%%%%%%
\begin{lemma}\label{LargeLemma}
Fix $T>0$. Consider $\zeta^{in}$ satisfies \eqref{ConditionIntial} and assume that the coagulation rate $\mathcal{A}$ satisfies \eqref{SumK}. Suppose $\mathcal{S}$ and $\mathcal{R}$ enjoy assumptions  \eqref{RateInjection} and \eqref{RateRemovable}, respectively.  Then for every $n> 1$, we have
\begin{equation}\label{C(T)}
\sup_{t\in [0, T]}\int_0^n \sigma_1 (p) \zeta_n(t, p)\ dp  \le \Xi(T),
\end{equation}
\begin{align}\label{C(T1)}
\int_0^T \int_0^n\int_{n-p}^n & \sigma_1 (p)  \mathcal{A}_n(p, q) \ \zeta_n(s, p) \ \zeta_n(s, q) \ dq dp ds \le \Xi(T),
\end{align}
and
\begin{align}
\int_0^t \int_0^n \sigma_1(p) \mathcal{R}_n(p) \zeta_n(s, p)\  dp ds \le \Xi(T),
\end{align}
where $\Xi(T)$ (depending on $T$) is a positive constant and the $\sigma_1$ is a convex function and satisfies \eqref{Convexp1} and \eqref{Convexp2}.
\end{lemma}

\begin{proof} We set $\omega (p) :=\sigma_1(p) \chi_{(0, n)}(p)$, and inserting it into \eqref{TVWSolution} to obtain
\begin{align}\label{Large1}
&\int_0^n \sigma_1(p) \zeta_n(t, p) dp \nonumber\\
=& \int_0^n \sigma_1(p) \zeta_n^{in}(p)dp + \frac{1}{2} \int_0^t \int_0^n\int_0^{n} H_{\sigma_1, n} (p, q)   \mathcal{A}_n(p, q) \zeta_n(s, p) \zeta_n(s, q) \ dq dp ds\nonumber\\
& + \int_0^t\int_0^n \sigma_1(p) \mathcal{S}_n(p) dp ds - \int_0^t \int_0^n \sigma_1(p) \mathcal{R}_n(p) \zeta_n(s, p)  dp ds,
\end{align}
where
\begin{align}\label{Large2}
 H_{\sigma_1, n} (p, q)= \sigma_1 (p+q)\chi_{(0, n)}(p+q)-[\sigma_1 (p)+\sigma_1 (q)].
 \end{align}
By using \eqref{Convexp2} and \eqref{Convexp3} into \eqref{Large1}, we have
\begin{align}\label{large3}
\int_0^n \sigma_1(p) \zeta_n(t, p) dp \le & \Gamma_1 + \Gamma_3 t + \frac{1}{2}\int_0^t [ \mathcal{E}_n(s) + \mathcal{F}_n(s)]ds  \nonumber\\
&- \int_0^t \int_0^n \mathcal{R}_n(p) \zeta_n(s, p)  dp ds,
\end{align}
where
\begin{align*}
\mathcal{E}_n(s)= \int_0^n \int_0^{n-p} H_{\sigma_1, n}(p, q)  \mathcal{A}_n(p, q) \zeta_n(s, p) \zeta_n(s, q) dq dp,
\end{align*}
and
\begin{align*}
\mathcal{F}_n(s)= \int_0^n  \int_{n-p}^n  H_{\sigma_1, n} (p, q)  \mathcal{A}_n(p, q) \zeta_n(s, p) \zeta_n(s, q) dq dp.
\end{align*}
Next, we estimate $\mathcal{E}_n(s)$, by using  \eqref{Convexp6} and \eqref{SumK}, as
\begin{align}\label{En}
\mathcal{E}_n(s) \le & 2 \int_0^n \int_0^{n} \mathcal{A}_n(p, q)  \frac{ p \sigma_1(q) + q \sigma_1(p)}{p+q}  \zeta_n(s, p) \zeta_n(s, q) dq dp \nonumber\\
\le & 4 A  \int_0^1 \int_0^{1-p}   \frac{p\sigma_1(q) }{p+q}\  \zeta_n(s, p) \zeta_n(s, q) dq dp  \nonumber\\
& + 2 A \int_{0}^1 \int_{1-p}^{n}   \{ p \sigma_1(q) + q \sigma_1(p) \}  \ \zeta_n(s, p) \zeta_n(s, q) dq dp \nonumber\\
& + 2 A \int_1^n \int_0^{n}  \{ p \sigma_1(q) + q \sigma_1(p) \} \   \zeta_n(s, p) \zeta_n(s, q) dq dp.
\end{align}
Let us estimate the first term on the right-hand side of \eqref{En}, by using the Lemma \ref{Uboundlemma} and the monotonicity of $\sigma_1$, as
\begin{align}\label{En1}
 4 A \int_0^1 \int_0^{1-p}    \sigma_1(q) \zeta_n(s, p) \zeta_n(s, q) dq dp
 \le 4 A \sigma_1(1) \Lambda^2(T).
\end{align}
Again, by using Lemma \ref{Uboundlemma} and the monotonicity of $\sigma_1$, the second term on the right-hand side of \eqref{En} can be evaluated, as
\begin{align}\label{En2}
 & 2 A \int_{0}^1 \int_{1-p}^{n}   \{ p \sigma_1(q) + q \sigma_1(p) \}  \ \zeta_n(s, p) \zeta_n(s, q) dq dp \nonumber\\
\le & 2 A  \int_0^1 \int_0^{n} \sigma_1(q)  \zeta_n(s, p) \zeta_n(s, q) dq dp  \nonumber\\
&+ 2 A \sigma_1(1) \int_0^1 \int_0^{n}  q \zeta_n(s, p) \zeta_n(s, q) dq dp  \nonumber\\
\le &   2 A \Lambda(T) \int_0^n   \sigma_1(p) \zeta_n(s, p) dp + 2 A \sigma_1(1) \Lambda^2(T).
\end{align}
Finally, we evaluate the last integral on the right-hand to \eqref{En}, by applying Lemma \ref{Uboundlemma}, as
\begin{align}\label{En3}
& 2 A \int_1^n \int_1^{n}  \{ p \sigma_1(q) + q \sigma_1(p) \} \ \zeta_n(s, p)\ \zeta_n(s, q) \ dq dp \nonumber\\
\le & 4 A \int_1^n \int_1^{n}  p \sigma_1(q) \zeta_n(s, p)\  \zeta_n(s, q)\ dq dp \le  4 A \Lambda(T) \int_0^n   \sigma_1(p) \zeta_n(s, p) dp.
\end{align}
Inserting \eqref{En1}, \eqref{En2} and \eqref{En3} into \eqref{En}, we obtain
\begin{align}\label{Enf}
\mathcal{E}_n(s) \le & 6 A \sigma_1(1) \Lambda^2(T) + 6 A \Lambda(T) \int_0^n   \sigma_1(p) \zeta_n(s, p) dp.
\end{align}
If $p + q > n$, then from \eqref{Large2}, we entail that 
\begin{align}\label{HValue}
H_{\sigma_1}(p, q)=-\sigma_1(p)-\sigma_1(q).
\end{align}
Using \eqref{HValue}, $\mathcal{F}_n(s)$ can be rewritten as
\begin{align}\label{Fn}
\mathcal{F}_n(s)= - 2  \int_0^n  \int_{n-p}^{n} \sigma_1(p)  \mathcal{A}_n(p, q)\ \zeta_n(s, p) \zeta_n(s, q)\ dq dp \le  0.
\end{align}
Since $\zeta_n \ge 0$ and $\mathcal{R}_n \ge 0$, then 
\begin{align}\label{REstimate}
- \int_0^t \int_0^n \sigma_1(p) \mathcal{R}_n(p)\ \zeta_n(s, p)  dp ds \le 0.
\end{align}
Inserting (\ref{Enf}), (\ref{Fn}) and \eqref{REstimate} into (\ref{Large1}) and introducing
\begin{align*}
  \mathcal{V}_n(t) :=  & \int_0^n \sigma_1(p) \zeta_n(t, p)dp   + \int_0^t \int_0^n \sigma_1(p) \mathcal{R}_n(p) \zeta_n(s, p)  dp ds  \nonumber\\
 &+   \int_0^t \int_0^n\int_{n-p}^{n} \sigma_1(p)  \mathcal{A}_n (p, q)\ \zeta_n(s, p) \zeta_n(s, q) dq dp ds,
\end{align*}
we end up with
\begin{align*}
  \mathcal{V}_n(t)  \le  \Gamma_1 +  \Gamma_3 t + 6 A \sigma_1(1) \Lambda^2(T) + 6 A \Lambda(T) \int_0^n   \sigma_1(p) \zeta_n(s, p) dp.
\end{align*}
Finally, by applying the Gronwall inequality, it can infer that
\begin{align*}
  \mathcal{V}_n(t) \le  \Xi(T),
\end{align*}
where $\Xi(T) :=\bigg(\Gamma_1+ \Gamma_3 T + \frac{\Gamma_3}{6A \Gamma(T) } + 6 A \sigma_1(1) \Lambda^2(T)  \bigg) e^{ 6 A \Lambda(T) T} $,
which completes the proof of Lemma \ref{LargeLemma}.
\end{proof}
%%%%%%%%%%%%%%%%%%%%%%%%%%%%%%%%%%%%%%%%%%
%%%%%%%%%%%%%%%%%%%%%%%%%%%%%%%%%%%%%%%%%%

%%%%%%%%%%%%%%%%%%%%%%%%%%%%%%%%%%%%%%%%
%%%%%%%%%%%%%%%%%%%%%%%%%%%%%%%%%%%%%%%%

The next lemma provides the equi-integrability for the family of solutions $\{ \zeta_n\}_{n > 1}$.
\subsection{Equi-integrability}
\begin{lemma}\label{LemmaUniformIntegrability}
 Fix $T>0$. Consider $\zeta^{in}$ satisfies \eqref{ConditionIntial} and assume that the coagulation rate $\mathcal{A}$ satisfies \eqref{EstimateCoagRate}. Suppose $\mathcal{S}$ and $\mathcal{R}$ enjoy assumptions  \eqref{RateInjection} and \eqref{RateRemovable}, respectively. Then, for $ \lambda >1$, there is a positive constant $C(T, \lambda)$ such that
\begin{align*}
 \sup_{t\in [0,T]}\int_0^{\lambda} \sigma_2(  \zeta_n(t, p))dp \le C(T, \lambda),
\end{align*}
where $\sigma_2 $ is a convex function and  $\sigma'_2$ is concave and satisfies \eqref{Convexp1} and \eqref{Convexp2}.
\end{lemma}

\begin{proof} Let $ \lambda >1 $. Then from \eqref{TSCEs}, and the non-negativity of $\zeta_n$, $\mathcal{A}_n$ and $\mathcal{R}_n$,  it gives
\begin{align}\label{Equintp2}
 \frac{d}{dt}\int_0^{\lambda} \sigma_2( \zeta_n(t, p)) dp 
 \le & \frac{1}{2} \int_0^{\lambda} \int_0^p \sigma_2'( \zeta_n(t, p))  \mathcal{A}_n(p-q, q) \zeta_n(t, p-q) \zeta_n(t, q) dq dp\nonumber\\
&+\int_0^{\lambda}  \sigma_2'(\zeta_n(t, p)) \mathcal{S}_n(p) dp.
\end{align}
Use of Fubini's theorem, substitute the transformation $p-q=p'$ \&  $q=q'$, and \eqref{Convexp3} into \eqref{Equintp2}, we have
\begin{align}\label{Equintp3}
 \frac{d}{dt}\int_0^{\lambda} \sigma_2(\zeta_n(t, p)) dp 
 \le &\frac{1}{2} \int_0^{\lambda} \int_0^{\lambda-q} \sigma_2'( \zeta_n(t, p+q))  \mathcal{A}_n(p, q) \zeta_n(t, p) \zeta_n(t, q) dp dq\nonumber\\
&+\int_0^{\lambda}   [ \sigma_2(\zeta_n(t, p)) + \sigma_2 (\mathcal{S}_n(p))  ] dp \nonumber\\
\le &\frac{1}{2} \int_0^{\lambda} \int_0^{\lambda-q} \sigma_2'( \zeta_n(t, p+q))  \mathcal{A}_n(p, q) \zeta_n(t, p) \zeta_n(t, q) dp dq\nonumber\\
&+\int_0^{\lambda}    \sigma_2(\zeta_n(t, p))   dp + \Gamma_4.
\end{align}
We next evaluate the following term  by applying \eqref{EstimateCoagRate}, \eqref{Convexp4} and Lemma \ref{Uboundlemma}, as 
\begin{align}\label{est1}
& \frac{1}{2} \int_0^{\lambda}  \int_0^{\lambda-q}  \sigma_2'(\zeta_n(t, p+q))   \mathcal{A}_n (p, q) \zeta_n(t, p) \zeta_n(t, q) dp dq\nonumber\\
\le & \frac{1}{2} A (1+\lambda)^2 \int_0^{\lambda} \int_0^{\lambda-q}    \sigma_2'(\zeta_n(t, p+q)) \zeta_n(t, p) \zeta_n(t, q) dp dq\nonumber\\
\le & \frac{1}{2} A (1+\lambda)^2  \int_0^{\lambda} \int_0^{\lambda-q}   [\sigma_2(\zeta_n(t, p+q))+\sigma_2 (\zeta_n(t, q))] \zeta_n(t, p) dp dq\nonumber\\
\le &  C_1 (T, \lambda) \int_0^{\lambda} \sigma_2 (\zeta_n(t, p))dp,
\end{align}
where $C_1(T, \lambda ) := A (1+ \lambda )^2 \Lambda(T)$. 
By using  \eqref{est1} into \eqref{Equintp3}, we get
\begin{align}
\frac{d}{dt}\int_0^{\lambda} \sigma_2(\zeta_n(t, p))dp\le  C_2(T, \lambda )\int_0^{\lambda} \sigma_2(\zeta_n(t, p))dp+ \Gamma_4,
\end{align}
where $C_2(T, \lambda ):= C_1(T, \lambda)+ 1$.
Finally, the Gronwall inequality gives
\begin{align}
\int_0^{\lambda} \sigma_2(\zeta_n(t, p)) dp \le C(T, {\lambda}),
\end{align}
where $C(T, {\lambda})$ is a constant depending on $T$ and $ \lambda $. This completes the proof of the Lemma \ref{LemmaUniformIntegrability}.
\end{proof}

%%%%%%%%%%%%%%%%%%%%%%%%%%%%%%%%%%%%%%%%%%%%%%%
%%%%%%%%%%%%%%%%%%%%%%%%%%%%%%%%%%%%%%%%%%%%%%%
\subsection{Time Equi-continuity}
%%%%%%%%%%%%%%%%%%%%%%%%%%%%%%%%%%%%%%%%%%%%%%
%%%%%%%%%%%%%%%%%%%%%%%%%%%%%%%%%%%%%%%%%%%%%%
\begin{lemma}\label{Equicontinuityweaksense}
 Fix $T>0$. Consider $\zeta^{in}$ satisfies \eqref{ConditionIntial} and assume that the coagulation rate $\mathcal{A}$ satisfies the sum kernel stated in \eqref{EstimateCoagRate}. Suppose $\mathcal{S}$ and $\mathcal{R}$ enjoy assumptions  \eqref{RateInjection} and \eqref{RateRemovable}, respectively. For any $ \lambda  >1$, there is a positive constant $C_5(T, \lambda)$ depending on $T$ and $\lambda$ such that
\begin{align*}
\int_0^{\lambda}   | \zeta_n(t, p)- \zeta_n(s, p) | dp  \le C_5(T, \lambda)(t-s),
\end{align*}
for every $n > 1$ and $0 \le s \le t \le T$.
\end{lemma}

\begin{proof}
Let $T>0$ and $\lambda >1$. For $0 \le s \le t \le T$, we evaluate the following integral as
\begin{align}\label{Equicontinuity1}
 \int_0^{\lambda}    |\zeta_n(t, p)- \zeta_n(s, p)| dp 
 \le &  \int_s^t \int_0^{\lambda}  \bigg|  \frac{\partial \zeta_n}{\partial t}(\tau, p) \bigg| dp d\tau \nonumber\\
 \le  &  \int_s^t  \bigg[   \int_0^{\lambda}  \mathcal{C}_a^n(\zeta_n (\tau, p) ) dp  + \int_0^{\lambda}   \mathcal{C}_d^n(\zeta_n (\tau, p) ) dp  \nonumber\\
  &+ \int_0^{\lambda}   \mathcal{S}_n(p)  dp +\int_0^{\lambda}    \mathcal{R}_n(p) \zeta_n(\tau, p) dp  \bigg] d\tau.
\end{align}
An application of the Fubini theorem, we estimate
\begin{align}\label{Equicontinuity2}
  \int_0^{\lambda}  \mathcal{C}_a^n(\zeta_n (\tau, p) ) dp  \le  \frac{1}{2} \int_0^{\lambda}   \mathcal{C}_d^n(\zeta_n (\tau, p) ) dp.
\end{align}
Let us now evaluate the following integral, by using Fubini's theorem, \eqref{EstimateCoagRate} and Lemma \ref{Uboundlemma}, as
\begin{align}\label{Equicontinuity3}
 \int_s^t \int_0^{\lambda}     \mathcal{C}_d^n(\zeta_n (\tau, p) ) dp d\tau 
\le &  A  \int_s^t \int_0^{\lambda}  \int_{0}^{n}  (1+p) (1+q)  \zeta_n(\tau, p) \zeta_n(\tau, q) dq dp d\tau \nonumber\\
\le & A \Lambda^2(T)  (t-s).
\end{align}
Since $\mathcal{S} \in L^1_{0, 1} (\mathbb{R}_{>0})$, then the third integral can be estimated as
\begin{align}\label{Equicontinuity4}
  \int_s^t \int_0^{\lambda}  \mathcal{S}_n(p)  dp  d\tau 
   \le  \| \mathcal{S} \|_{L^1_{0, 1} (\mathbb{R}_{>0}) } (t-s).
\end{align}
Finally, the last term can be evaluated, by applying Lemma \ref{Uboundlemma} $(ii)$ as
\begin{align}\label{Equicontinuity5}
 \int_s^t \int_0^{\lambda}   \mathcal{R}_n(p) \zeta_n (\tau, p) dp d\tau
  \le  \|\mathcal{S}\|_{L^1_{0, 1}(\mathbb{R}_{>0})}  (t-s).
\end{align}

Using \eqref{Equicontinuity2}, \eqref{Equicontinuity3}, \eqref{Equicontinuity4} and \eqref{Equicontinuity5} into \eqref{Equicontinuity1}, we get
\begin{align}\label{Equicontinuity6}
\int_0^{\lambda}    | \zeta_n(t, p) - \zeta_n(s, p)| dp \le  C_5(T) (t-s),
\end{align}
where
\begin{align*}
C_5(T, \lambda) : = \bigg[ \frac{3}{2} A  \Lambda^2(T) + 2 \| \mathcal{S} \|_{L^1_{0, 1} (\mathbb{R}_{>0}) }   \bigg].
\end{align*}
This completes the proof of the Lemma \ref{Equicontinuityweaksense}.
\end{proof}

From the Lemma \ref{Equicontinuityweaksense}, we ensure that the family $\{ \zeta_n \}_{n>1}$ is strongly continuous in $L^1(0, \lambda)$. Since we know that strongly continuous implies weakly continuous. Thus,  the family $\{ \zeta_n \}_{n>1}$ is also weakly continuous, i.e., $\{ \zeta_n \}_{n>1} \subset \mathcal{C}([0, T]_w: L^1(0, \lambda))$,
more precisely, for every $\omega \in L^{\infty }(\mathbb{R}_{>0})$, we have
\begin{align}\label{weakcontinuity}
\int_0^{\lambda}  \omega(p) | \zeta_n(t, p)- \zeta_n(s, p) | dp  \le \| \omega \|_{L^{\infty }(\mathbb{R}_{>0})}C_5(T, \lambda)(t-s).
\end{align}

We are now in a position to complete the proof of Theorem \ref{TheoremSCEs} in the next subsection.

%%%%%%%%%%%%%%%%%%%%%%%%%%%%%%%%%%%%%%%%%%%%%%%
%%%%%%%%%%%%%%%%%%%%%%%%%%%%%%%%%%%%%%%%%%%%%%%
\subsection{Convergence of integrals}
%%%%%%%%%%%%%%%%%%%%%%%%%%%%%%%%%%%%%%%%%%%%%%
%%%%%%%%%%%%%%%%%%%%%%%%%%%%%%%%%%%%%%%%%%%%%%
\begin{proof} \emph{of Theorem} \ref{TheoremSCEs} (i):
From the De la Vall\`{e}e Poussin theorem \ref{TheoremDelaVallee}, Lemma \ref{LemmaUniformIntegrability}, \eqref{weakcontinuity} and then using the Dunford-Pettis theorem and a variant of the Arzel\`{a}-Ascoli theorem, see \cite{Vrabie:1995}, we conclude that $(\zeta_n)_{n>1}$ is relatively compact in $\mathcal{C}([0, T]_w; L^1(0, \lambda))$ for each $T>0$. Thus, there exists a subsequence of $(\zeta_n)_{n>1}$ (not relabeled) and a non-negative function $\zeta \in \mathcal{C}([0, T]_w; L^1(0, \lambda))$ such that
\begin{align}\label{weakconvergence}
 \zeta_n \to \zeta \ \ \ \text{in} \ \ \mathcal{C}([0, T]_w: L^1(0, \lambda)).
\end{align}
Then, proceeding the same steps as in \cite{Barik:2021}, we can improve the convergence \eqref{weakconvergence} to 
\begin{align}\label{weakconvergence1}
 \zeta_n \to \zeta \ \ \ \text{in} \ \ \mathcal{C}([0, T]_w: L_{0,1}^1(\mathbb{R}_{>0}) )\ \ \text{for each $T>0$},
\end{align}
by applying Lemma \ref{LargeLemma} and \eqref{weakconvergence}.  \\
% Next, we claim that
%\begin{align}\label{Convergencestrongly}
% \zeta \in   \mathcal{C} ( [0, T]; L_{0, 1}^1 ( \mathbb{R}_{>0}  )).
%\end{align}
%Let $t\ge s \ge 0$ and $\omega \in L^{\infty}( \mathbb{R}_{>0} )$. Since from \eqref{weakconvergence1}, we find $\{ \zeta_n(t) - \zeta_n(s) \}$ converges weakly to $\{ \zeta(t) - \zeta(s) \}$ in $L^1( \mathbb{R}_{>0} ) $, then we can pass to the limit $n \to \infty$ in \eqref{weakcontinuity} and obtain that $\zeta$ also satisfies \eqref{weakcontinuity} from which we infer that
% \begin{align}\label{eq1}
%& \|\zeta(t)- \zeta(s) \|_{L^1(\mathbb{R}_{>0}  )} \nonumber\\
%  = & \sup_{ \omega \in L^{\infty} ( \mathbb{R}_{>0} ) } \bigg\{  \frac{1}{\|\omega \|_{L^{\infty}( \mathbb{R}_{>0} )   }} \bigg| \int_0^{\infty} \{ \zeta( t, p)- \zeta(s, p) \} \omega(p) dp \bigg| \bigg\} \nonumber\\
% \le  & C_5(T, \lambda) (t-s).
 % \end{align}
 % This proves the claim \eqref{Convergencestrongly}.  
  In order to complete the proof of Theorem \ref{TheoremSCEs} (i), we need to show that $\zeta$ is actually a solution to \eqref{SCEs}--\eqref{Initialdata} in a weak sense. Thus, it remains to verify all the truncated integrals in \eqref{TSCEs} converge weakly to the original integrals in \eqref{SCEs}, respectively.\\
  
 Suppose $\omega \in L^{\infty} (\mathbb{R}_{>0})$ with compact support included in $\mathbb{R}_{>0}$ for some $ \lambda>1$. By \eqref{TVWSolution}, we have
\begin{align}\label{Barik31}
\int_0^n \omega(p) [\zeta_n(t, p)- \zeta_n^{in}(p )] dp  = \mathcal{J}_1^n(t) + \mathcal{J}_2^ (t)+ \mathcal{J}_3^n(t)- \mathcal{J}_4^n(t)- \mathcal{J}_5^n(t),
\end{align}
where
\begin{align*}
\mathcal{J}_1^n (t) := &  \int_0^t \int_{0}^{\lambda} \mathcal{C}_a^n(\zeta_n (s, p) ) dp ds, \  
\mathcal{J}_2^n (t) :=   \int_0^t \int_{0}^{\lambda} \mathcal{C}_d^n(\zeta_n (s, p) ) dp ds, \\
\mathcal{J}_3^n (t) := & \int_0^t  \int_{0}^{n} \omega(p) \mathcal{S}_n(p)   dp ds, \ 
\mathcal{J}_4^n (t) :=  \int_0^t  \int_0^{\lambda}  \omega(p) \mathcal{R}_n(p)  \zeta_n(s, p) dp ds \\
\mathcal{J}_5^n (t) := & \int_0^t  \int_{\lambda}^n   \omega(p) \mathcal{R}_n(p)  \zeta_n(s, p) dp ds.
\end{align*}
 This is now a standard way to prove this convergence of integrals for the Smoluchowski coagulation equation by applying Lemma \ref{Uboundlemma}, Lemma \ref{LargeLemma} and \eqref{Convexp1}, see \cite{Barik:2018, Giri:2012, Laurencot:2002, Laurencot:2015, Stewart:1989}. We thus omit the proof of the convergence of integrals for the $\mathcal{C}_a^n(\zeta_n (t, p) )$ and $\mathcal{C}_d^n(\zeta_n (t, p) ) $.
  \\
The convergence of  $\mathcal{J}_3^n(t)$ can be easily shown by using the convergence of $\mathcal{S}_n$ to $\mathcal{S}$ and the Lebesgue dominated convergence theorem.\\

From \eqref{RateInjection} and $\omega \in L^{\infty} (\mathbb{R}_{>0})$, we find
\begin{align*}
\omega(p) \mathcal{R}_n(p)   \le  \| \omega\|_{L^{\infty}(0, \lambda) } k (1+ \lambda) \in L^{\infty}(0, \lambda).
\end{align*}
Since $\zeta_n \to \zeta \in \mathcal{C}([0,T]_w; L^1( \mathbb{R}_{>0} ))$ as $n \to \infty$, and $\omega(p) \mathcal{R}_n(p) \in L^{\infty}(0, \lambda) $, then by \cite{Laurencot:2002} Lemma 4.3, we have
\begin{align}\label{J41}
\lim_{n \to \infty}\mathcal{J}_4^n (t) 
= \int_0^t  \int_0^{\lambda}  \omega(p) \mathcal{R}(p)  \zeta(s, p) dp ds.
\end{align}

Now, $\mathcal{J}_5^n(t)$ can be evaluated, by applying \eqref{RateInjection}, Lemma \ref{LargeLemma} and \eqref{Convexp2}, as 
\begin{align*}
|\mathcal{J}_5^n (t)| % = & \bigg|\int_0^t  \int_{\lambda}^n   \omega(p) \mathcal{R}_n(p)  \zeta_n(s, p) dp ds \bigg| \nonumber\\
\le & \| \omega\|_{L^{\infty}(\mathbb{R}_{>0})  }   \bigg| \int_0^t   \int_{\lambda}^n  \mathcal{R}_n(p)  \zeta_n(s, p) dp ds\bigg| \nonumber\\
\le & \| \omega\|_{L^{\infty} (\mathbb{R}_{>0}) }    \frac{\lambda} {\sigma_1(\lambda)} \bigg| \int_0^t  \int_{\lambda}^n   \sigma_1(p)  \mathcal{R}_n(p)   \zeta_n(s, p) dp ds    \bigg|\nonumber\\
\le & \| \omega\|_{L^{\infty}  (\mathbb{R}_{>0}) }   \frac{\lambda} {\sigma_1(\lambda)} \Xi(T).
\end{align*}

As limit $\lambda  \to \infty$, then 
\begin{align}\label{J51}
|\mathcal{J}_5^n (t)|   \to 0.
\end{align}

Similarly, we can show that the following term goes to $0$ as $\lambda \to \infty$
\begin{align}\label{J52}
 \bigg|\int_0^t  \int_{\lambda}^{\infty}   \omega(p) \mathcal{R}(p)  \zeta(s, p) dp ds \bigg| 
\le & \| \omega\|_{L^{\infty}(\mathbb{R}_{>0})  }   \bigg| \int_0^t    \int_{\lambda}^{\infty}  \mathcal{R}(p)  \zeta(s, p) dp ds\bigg| \nonumber\\
\le & \| \omega\|_{L^{\infty} (\mathbb{R}_{>0}) }    \frac{\lambda} {\sigma_1(\lambda)} \bigg| \int_0^t  \int_{\lambda}^{\infty} \sigma_1(p) \mathcal{R}(p)     \zeta(s, p) dp ds    \bigg|\nonumber\\
\le & \| \omega\|_{L^{\infty}  (\mathbb{R}_{>0}) }    \frac{\lambda} {\sigma_1(\lambda)} \Xi(T).
\end{align}
Finally, combining \eqref{J41}, \eqref{J51}, \eqref{J52} and taking first limit $\lambda \to \infty$ and then $n \to \infty$, we end up with
\begin{align}\label{last}                                                        
& \lim_{n\to \infty} \bigg\{
 \frac{1}{2} \int_0^t \int_0^n\int_0^{n} H_{\omega, n}(p, q)   \mathcal{A}_n(p, q)  \zeta_n(s, p) \zeta_n(s, q) dq dp ds\nonumber\\ 
 &+ \int_0^t\int_0^n  [ \mathcal{S}_n(p)  - \mathcal{R}_n(p) \zeta_n(s, p) ] dp ds \bigg\}\nonumber\\
=&  \frac{1}{2}\int_0^t \int_0^{\infty} \int_{0}^{\infty}\tilde{\omega}(p, q) \mathcal{A}(p, q)  \zeta(s, p) \zeta(s, q) dq dp ds\nonumber\\
& + \int_0^t \int_0^{\infty}  [ \mathcal{S}(p)  - \mathcal{R}(p) \zeta(s, p) ] dp ds.
\end{align}

Using once more the weak convergence $\zeta_n \longrightarrow \zeta ~\mbox{in} ~\mathcal{C}_w([0,T]; L^1(\mathbb{R}_{>0} ))$ into \eqref{Barik31} and applying \eqref{last}, we conclude
\begin{align*}
\int_0^{\infty} \omega(p) & [ \zeta(t, p)- \zeta^{in}(p) ] dp = \lim_{n \to \infty} \int_0^n \omega(p) [ \zeta_n(t, p)- \zeta_n^{in}(p)] dp \nonumber\\
 = & \lim_{n\to \infty} \bigg\{
 \frac{1}{2} \int_0^t \int_0^n\int_0^{n} H_{\omega, n}(p, q)   \mathcal{A}_n(p, q)  \zeta_n(s, p) \zeta_n(s, q) dq dp ds\nonumber\\ 
 &+ \int_0^t\int_0^n  [ \mathcal{S}_n(p)  - \mathcal{R}_n(p) \zeta_n(s, p) ] dp ds \bigg\}\nonumber\\
=&  \frac{1}{2}\int_0^t \int_0^{\infty} \int_{0}^{\infty}\tilde{\omega}(p, q) \mathcal{A}(p, q)  \zeta(s, p) \zeta(s, q) dq dp ds\nonumber\\
& + \int_0^t \int_0^{\infty}  [ \mathcal{S}(p)  - \mathcal{R}(p) \zeta(s, p) ] dp ds,
\end{align*}
for every $\omega \in L^{\infty}(\mathbb{R}_{>0})$. This confirms that $ \zeta$ satisfies \eqref{definition}. Thus, $ \zeta$ is a weak solution to \eqref{SCEs}--\eqref{Initialdata}. Next, we need to show that $ \zeta$ satisfies \eqref{MCPS}. It can easily be proved  by using \eqref{PropMassbound}, Lemma \ref{LargeLemma}, \eqref{Convexp1}, and \eqref{Cond R S}. This completes the proof of the Theorem \ref{TheoremSCEs} (i).

\end{proof}

\begin{proof} \emph{of Theorem} \ref{TheoremSCEs} (ii): This part can be proved by using similar steps as in the proof of Theorem \ref{TheoremSCEs} (i). But, for completeness, we provide a short proof. The family $(\zeta_n)_{n>1}$ is relatively compact in $\mathcal{C}([0, T]_w; L^1(0, \lambda))$ for each $T>0$, can be shown by applying de la Vall\`{e}e Poussin theorem, Lemma \ref{LemmaUniformIntegrability}, \eqref{weakcontinuity}, the Dunford-Pettis theorem and a variant of the Arzel\`{a}-Ascoli theorem, see \cite{Vrabie:1995}. 
More precisely,
\begin{align}\label{weakconvergencep}
 \zeta_n \to \zeta \ \ \ \text{in} \ \ \mathcal{C}([0, T]_w: L^1(0, \lambda))\ \ \text{for each }\ T>0.
\end{align}
 The convergence \eqref{weakconvergencep} can be improved, by using Lemma \ref{Uboundlemma} to
\begin{align}\label{weakconvergencep1}
 \zeta_n \to \zeta \ \ \ \text{in} \ \ \mathcal{C}([0, T]_w: L^1(\mathbb{R}_{>0} )). 
\end{align}
%Furthermore, in the same fashion, we can show that
%\begin{align}\label{Convergencestronglyp}
 %\zeta \in   \mathcal{C} ( [0, T]; L^1 ( \mathbb{R}_{>0}  )).
%\end{align}
Next, our aim is to show all the approximated integrals in \eqref{TSCEs} converge weakly to the original integrals in \eqref{SCEs}, respectively, for the product type kernel as in \eqref{ProductK}.

The convergence of  $\mathcal{J}_1^n(t)$, $\mathcal{J}_2^n(t)$, $\mathcal{J}_3^n(t)$, and $\mathcal{J}_4^n(t)$ can be shown similarly as for the linear kernel with little modification. Thus, we omit the detailed calculations of the convergence of these terms.

Now, $\mathcal{J}_5^n(t)$ can be evaluated, by applying \eqref{RateInjection}, Lemma \ref{LargeLemma}, and \eqref{Convexp2}, as 
\begin{align*}
|\mathcal{J}_5^n (t)| = & \bigg|\int_0^t  \int_{\lambda}^n   \omega(p) \mathcal{R}_n(p)  \zeta_n(s, p) dp ds \bigg| \nonumber\\
\le & \| \omega\|_{L^{\infty}(\mathbb{R}_{>0})  }  k  \frac{1}{(1+\lambda)^{ 1- \alpha} } \int_0^t   \int_{\lambda}^n  (1+p)  \zeta_n(s, p) dp ds \nonumber\\
\le & \| \omega\|_{L^{\infty}(\mathbb{R}_{>0})  }  k \frac{1}{(1+\lambda)^{1- \alpha} }  \Lambda(T)   t.
\end{align*}

As limit $\lambda  \to \infty$, then 
\begin{align}\label{PJ51}
|\mathcal{J}_5^n (t)|   \to 0.
\end{align}

Similarly, we can show that the following term goes to $0$ as $\lambda \to \infty$
\begin{align}\label{PJ52}
 \bigg|\int_0^t  \int_{\lambda}^{\infty}   \omega(p) \mathcal{R}(p)  \zeta(s, p) dp ds \bigg| 
\le & \| \omega\|_{L^{\infty}(\mathbb{R}_{>0})  }  k  \frac{1}{(1+\lambda)^{1-\alpha} } \int_0^t   \int_{\lambda}^n  (1+p)  \zeta(s, p) dp ds \nonumber\\
\le & \| \omega\|_{L^{\infty}(\mathbb{R}_{>0})  }  k \frac{1}{(1+\lambda)^{1-\alpha} }    \| \zeta \|_{L^1_{0, 1}(\mathbb{R}_{>0})} t.
\end{align}
Finally, combining \eqref{PJ51} and \eqref{PJ52} and taking first limit $\lambda \to \infty$ and then $ n \to \infty$, we see that
$\zeta$ satisfies \eqref{definition}. Therefore, $\zeta$ is a weak solution to \eqref{SCEs}--\eqref{Initialdata} in the sense of \eqref{definition} which completes the proof of Theorem \ref{TheoremSCEs}.
\end{proof}

\section{Uniqueness of weak solutions}
In this section, our aim is to prove Theorem \ref{UniquenessTheoremCRBK}. 
\begin{proof} We prove the Theorem \ref{UniquenessTheoremCRBK} by the method of the contrary. Suppose \eqref{SCEs}--\eqref{Initialdata} has no unique solution. 
 Assume that $\zeta$ and $\eta$ be two distinct solutions (in a weak sense) to \eqref{SCEs} with $\zeta^{in} = \eta^{in}$. Consider $\Delta := \zeta - \eta $ and $\text{sign}(\Delta) := L$. Let $\omega(p)  := \max\{1, p^{1/2}\}$, then consider the following integral:
\begin{align}\label{Uniqueness1}
& \frac{d}{dt} \int_0^{\infty} \omega(p) |\Delta(t, p)| dp =  \int_0^{\infty} \omega(p) L(t, p) \frac{\partial \Delta(t, p) }{\partial t}  dp \nonumber\\
= & \frac{1}{2}\int_0^{\infty} \int_0^{\infty} \mathcal{A}(p, q) L(t, p, q) \{ \zeta(t, p) \zeta(t, q) - \eta(t, p) \eta(t, q)  \}  dq dp \nonumber\\
& - \int_0^{\infty}\omega(p) \mathcal{R}(p) L(t, p) \Delta(t, p) dp \nonumber\\ 
= & \frac{1}{2}\int_0^{\infty} \int_0^{\infty} \mathcal{A}(p, q) L(t, p, q) \{ \zeta(t, p) ( \zeta(t, q) -  \eta(t, q) ) + \eta(t, q) (\zeta(t, p) -  \eta(t, p) ) \}  dq dp \nonumber\\
& - \int_0^{\infty}\omega(p) \mathcal{R}(p) L(t, p) \Delta(t, p) dp\nonumber\\
= & \frac{1}{2}\int_0^{\infty} \int_0^{\infty} \mathcal{A}(p, q) L(t, p, q) \zeta(t, p)   \Delta(t, q)     dq dp \nonumber\\
 & + \frac{1}{2}\int_0^{\infty} \int_0^{\infty} \mathcal{A}(p, q) L(t, p, q)  \eta(t, q) \Delta(t, p)   dq dp \nonumber\\
& - \int_0^{\infty}\omega(p) \mathcal{R}(p) |\Delta(t, p)| dp,
\end{align}
where 
\begin{align*}
 L(t, p, q) := \omega(p + q) \zeta(t, p+q) - \omega(p) \zeta(t, p) - \omega(q) \zeta(t, q).
\end{align*}
We can see that 
\begin{align}\label{Bound11}
 L(t, p, q) \Delta(t, q) \le \{ \omega(p+q)  + \omega(p) - \omega(q) \} | \Delta(t, q) | 
\end{align}
and
\begin{align}\label{Bound12}
L(t, p, q) \Delta(t, p) \le \{ \omega(p+q)  - \omega(p) + \omega(q) \} | \Delta(t, p) |.
\end{align}
Let us now discuss different cases:\\
\textbf{Case-1} For $p+q < 1$, by using \eqref{Bound11} and \eqref{UniqueThm}, it can see 
\begin{align}\label{Boundcase1a}
\mathcal{A}(p, q)  L(t, p, q) \Delta(t, q) \le & A (1+p)^{1/2}(1+q)^{1/2} \{ 1+1-1 \} | \Delta(t, q) |   \nonumber\\
\le &  2 A | \Delta(t, q) | \omega(q),
\end{align}
and similarly, by applying \eqref{Bound12} and \eqref{UniqueThm}, we estimate
\begin{align}\label{Boundcase1b}
\mathcal{A}(p, q)  L(t, p, q) \Delta(t, p) \le  2A | \Delta(t, p) | \omega(p).  
\end{align}
\textbf{Case-2} For $p < 1 $, $q < 1$ and $p+q \geq 1$, by using \eqref{Bound11} and \eqref{UniqueThm}, we evaluate
\begin{align}\label{Boundcase2a}
\mathcal{A}(p, q)  L(t, p, q) \Delta(t, q) \le & A(1+p)^{1/2} (1+q)^{1/2} \{ (p+q)^{1/2} + 1 -1 \} | \Delta(t, q) | \nonumber\\
\le & 4 A p \omega(q) | \Delta(t, q) |,
\end{align}
and similarly, by applying \eqref{Bound12} and \eqref{UniqueThm}, we find
\begin{align}\label{Boundcase2b}
\mathcal{A} (p, q)  L(t, p, q) \Delta(t, p) \le  4A \omega(p) |\Delta(t, p)|.  
\end{align}
\textbf{Case-3} For $p < 1 $ and $q \ge 1$, the application of \eqref{Bound11} and \eqref{UniqueThm} gives
\begin{align}\label{Boundcase3a}
\mathcal{A}(p, q)  L(t, p, q) \Delta(t, q)  \le & A(1+p)^{1/2} (1+q)^{1/2} \{ (p+q)^{1/2} + 1 -  q^{1/2} \} | \Lambda(t, q) | \nonumber\\
\le &  2 A q^{1/2} \{ 1+q^{1/2} + 1 -  q^{1/2} \} | \Lambda(t, q) | \nonumber\\
\le & 4 A  \omega(q) | \Delta(t, q) |,
\end{align}
and similar to \eqref{Boundcase2b}, we evaluate
\begin{align}\label{Boundcase3b}
\mathcal{A}(p, q)  L(t, p, q) \Delta(t, p) \le  4A q \omega(p) |\Delta(t, p)|.  
\end{align}
\textbf{Case-4} For $p \ge 1 $ and $q < 1$, then by using \eqref{Bound11} and \eqref{UniqueThm}, we evaluate
\begin{align}\label{Boundcase4a}
\mathcal{A}(p, q)  L(t, p, q) \Delta(t, q) \le & A(1+p)^{1/2} (1+q)^{1/2} \{ 1+ 2 p^{1/2} -1 \} | \Delta(t, q) | \nonumber\\
\le & 4 A p \omega(q) | \Delta(t, q) |,
\end{align}
and similarly, by applying \eqref{Bound12} and \eqref{UniqueThm}, we find
\begin{align}\label{Boundcase4b}
\mathcal{A}(p, q)  L(t, p, q) \Delta(t, p) \le  4A \omega(p) |\Delta(t, p)|.  
\end{align}
\textbf{Case-5} For $p \ge 1$ and $q \ge 1$, then we estimate the following by using \eqref{Bound11} and \eqref{UniqueThm}, as
\begin{align}\label{Boundcase5a}
 \mathcal{A}(p, q)  L(t, p, q) \Delta(t, q)  \le & A(1+p)^{1/2} (1+q)^{1/2} \{ 2 p^{1/2} \} | \Delta(t, q) | \nonumber\\
  \le & 4A p \omega(q) | \Delta(t, q) |, 
\end{align}
and similarly, by applying \eqref{Bound12} and \eqref{UniqueThm}, we have
\begin{align}\label{Boundcase5b}
\mathcal{A}(p, q)  L(t, p, q) \Delta(t, p) \le  4 A q \omega(p) | \Delta(t, p) |.
\end{align}
Inserting above values from \eqref{Boundcase1a}--\eqref{Boundcase5b} into \eqref{Uniqueness1}, we obtain
\begin{align}\label{Uniqueness2}
& \frac{d}{dt} \int_0^{\infty} \omega(p) |\Delta(t, p)| dp \nonumber\\
\le & 2A \int_0^{1} \int_0^{1-p}  \{ \omega(q) |\Delta(t, q)| \zeta(t, p)  +  \omega(p) |\Delta(t, p)| \zeta(t, q)  \}  dq dp \nonumber\\
& +  4 A \int_0^{1} \int_{1-p}^1  \{  p \omega(q) | \Delta(t, q) \zeta(t, p)  +  \omega(p) |\Delta(t, p)| \zeta(t, q)  \}  dq dp \nonumber\\
& + 4 A \int_0^{1} \int_1^{\infty}  \{   \omega(q) | \Delta(t, q) | \zeta(t, p)  +  q \omega(p) |\Delta(t, p)| \zeta(t, q)  \}  dq dp \nonumber\\
& + 4 A \int_1^{\infty} \int_0^1  \{   p \omega(q) | \Delta(t, q) | \zeta(t, p)  + \omega(p) |\Delta(t, p)| \zeta(t, q)  \}  dq dp \nonumber\\
& + 4 A  \int_1^{\infty} \int_1^{\infty}  \{ p \omega(q) | \Delta(t, q) | \zeta(t, p)  + q \omega(p) | \Delta(t, p) | \zeta(t, q)  \}  dq dp \nonumber\\
\le & 36 A \Lambda(T) \int_0^{\infty} \omega(p) |\Lambda(t, p)| dp.
\end{align}
An application of Gronwall's inequality to \eqref{Uniqueness2} and using $\zeta^{in} = \eta^{in}$, we conclude that
\begin{align*}
 \int_0^{\infty} \omega(p) |\Lambda(t, p)| dp \le e^{36 \Gamma(T) A t} \int_0^{\infty} \omega(p) |\Lambda(0, p)| dp=0.
\end{align*}
This assures that
\begin{align*}
  |\Lambda(t, p)| =0,\ \ \text{for} \ \omega \in L^{\infty}(\mathbb{R}_{>0} ),
\end{align*}
i.e $\zeta = \eta $ a.e.. This violet to our assumption. Hence, this completes the proof of Theorem \ref{UniquenessTheoremCRBK}.
\end{proof}

%\section*{Acknowledgments}

%%%%%%%%%%%%%%%%%%%%%%%%%%%%%%%%%%
%%%%%%%%%%%%%%%%%%%%%%%%%%%%%%%%%%

\end{document}